\documentclass[11pt]{amsart}

\usepackage{biblatex}
\usepackage{epsfig}
\usepackage{amscd}
\usepackage[mathscr]{eucal}
\usepackage{amssymb}
\usepackage{amsxtra}
\usepackage{amsmath}
\usepackage[all]{xy}
\usepackage{color}

\bibliography{references}

\newtheorem{thm}{Theorem} 

\newtheorem{prop}{Proposition}[section]

\newtheorem{cor}[prop] {Corollary} 

\newtheorem{rem}[prop]{Remark} 

\newtheorem{lem}[prop]{Lemma} 

\newtheorem{que}[prop]{Question} 

\newtheorem{defn}[prop]{Definition}

\newtheorem{que-ed}{Question-editing}

\begin{document}
\title[Spectral Invariants]{Descent and $C^0$-rigidity of spectral invariants on monotone symplectic manifolds}

%\date{\today}
\author[Sobhan Seyfaddini]{Sobhan Seyfaddini}

\address{University of California Berkeley\\ Berkeley, CA 94720 \\USA}
\email{sobhan@math.berkeley.edu}

%\footnotetext[1]{The author}

\begin{abstract}

\noindent  Suppose that $(M^{2n}, \omega)$ is a closed, monotone symplectic manifold with $[\omega]|_{\pi_2(M)} = \lambda c_1|_{\pi_2(M)}$.  We show that if two Hamiltonians $H$ and $G$ vanish on a non-empty open set $U$ and the $C^0$ distance between $\phi^1_H$ and $\phi^1_G$ is sufficiently small, then the spectral invariants of $H$ and $G$ satisfy the following inequality:
 $$|c(a,G) - c(a,H)| \leq C \, d_{C^0}(\phi^1_G, \phi^1_H) + n \max(0, \lambda).$$
 As a corollary we obtain that spectral invariants descend from $\widetilde{ Ham_c}(M \setminus U)$ to $Ham_c(M\setminus U)$.
 
  We apply the above results to Hofer geometry and establish unboundedness of the Hofer diameter of $Ham_c(M\setminus U)$ for stably displaceable $U$.  We also answer a question of F. Le Roux about $C^0$-continuity properties of the Hofer metric.  
\end{abstract}

\maketitle

\section{Introduction and Main Results}
   Throughout this paper $(M, \omega)$ denotes a closed and connected symplectic manifold.   Any smooth Hamiltonian $H:[0,1] \times M \rightarrow \mathbb{R}$ induces a Hamiltonian path, or flow, $\phi^t_H : M \rightarrow M \text{ } (0\leq t \leq 1)$, by integrating the unique time-dependent vector field $X_H$ satisfying $dH_t = \iota_{X_H}\omega$, where $H_t(x) = H(t,x)$.  We denote the space of Hamiltonian paths by $PHam(M, \omega)$.  A Hamiltonian diffeomorphism is by definition any diffeomorphism obtained as the time-1 map of a Hamiltonian flow.  We denote by $Ham(M, \omega)$ the group of all Hamiltonian diffeomorphisms and by $\widetilde{Ham}(M, \omega)$ its universal cover.  We will eliminate the symplectic form $\omega$ from the above notations unless there is a possibility of confusion. 
   
    We equip $M$ with a distance $d$ induced by any Riemannian metric.  The $C^0$-topology on $Diff(M)$, the space of diffeomorphisms of $M$, is the topology induced by the distance $ \displaystyle  d_{C^0}(\phi, \psi):= \max_{x}d(x, \phi^{-1}\psi(x))$.  Please note that our definition yields a left-invariant metric; this property will be used in the proof of Theorem \ref{Main Theorem}.  Similarly, for paths of diffeomorphisms $\phi^t, \psi^t$ ($t\in [0,1]$) we define their $C^0$-distance by the expression $\displaystyle d_{C^0}^{path}(\phi^t, \psi^t):= \max_{t,x} d(x, (\phi^t)^{-1}\psi^t(x)).$    
   
  Each Hamiltonian $H \in C^{\infty}([0,1] \times M)$ has a set of spectral invariants
  $$\{c(a, H) \in \mathbb{R}: a \in QH^*(M) \setminus \{0\} \}.$$
  Spectral invariants were introduced by Viterbo in \cite{Vit92}.  In the form considered in this paper, they were constructed by Schwarz \cite{Sc00} on closed aspherical manifolds, and by Oh \cite{Oh05b} on all closed symplectic manifolds.  Simply stated,  $c(a, H)$ is the action level at which $a \in QH^*(M)$ appears in the Floer homology of $H$.  These invariants have been studied extensively and have had many interesting applications in symplectic geometry; see \cite{EP03, Gi10, Oh05a, Sc00}. Their construction and basic properties will be reviewed in Section \ref{Spectral Invariants}.  
  
  In \cite{Se11}, we obtained an estimate relating the difference between spectral invariants of two Hamiltonians to the $C^0$-distance of their flows.  In this article, we continue our research in this direction and obtain similar estimates relating the difference between spectral invariants of two Hamiltonians to the $C^0$-distance between \emph{the time$-1$ maps of their flows.}
  In Sections \ref{Estimates for spectral invariants}, \ref{Descent of spectral invariants}, \ref{Applications to Hofer geometry}, and \ref{Non-closed manifolds} we introduce our main results.  Spectral invariants and much of the notation of the paper are introduced in Section \ref{Spectral Invariants}.  Sections \ref{Proof of Main Theorem} and \ref{Proofs of applications to Hofer geometry} contain proofs of our results.  In Section \ref{Extension} we show that our results extend to some classes of compact symplectic manifolds with boundary.

  {\bf Acknowledgments:} The research that led to this article was initiated at the Institute for Advanced Study during the program on Symplectic Dynamics.  I would like to thank Helmut Hofer for his invitation, and the institute for its warm hospitality.  My trip to IAS would not have been possible without the support of my advisor, Alan Weinstein.  I want to thank him for this and for his helpful guidance over the past several years.  I would like to thank Michael Entov for several email communications and for drawing my attention to important examples.  I am very grateful to Leonid Polterovich for inviting me to the University of Chicago where I had the opportunity to present and discuss parts of this work with him, Strom Borman, Frol Zapolsky, and Lev Buhovsky.  I thank the four of them as well as Denis Auroux, Beno\^it Jubin, Michael Khanevsky, Fr\'ed\'eric Le Roux, and Dusa McDuff for helpful discussions and suggestions.  
  
  In an earlier version of this article, Theorems \ref{Th: infiniteness of diameter} and \ref{Th: Le Roux's question} were stated for displaceable sets.  Strom Borman pointed out to me that, in fact, those statements hold for stably displaceable sets.  He also, very kindly, provided me with a proof for Lemma \ref{Strom's Lemma}.  I am very thankful to him for all his help.
  
  \subsection{$C^0$-rigidity of spectral invariants} \label{Estimates for spectral invariants}
  Throughout this article, we assume that $U$ is a non-empty open subset of $M$.  Denote by $C^{\infty}_c([0,1] \times (M \setminus U))$ the set of smooth Hamiltonians with support compactly contained in $M \setminus U$.  We define $Ham_c(M \setminus U) = \{\phi^1_H : H \in C^{\infty}_c([0,1] \times (M \setminus U)) \}$.  Suppose $H, G \in C^{\infty}_c([0,1] \times (M \setminus U))$.  In \cite{Se11}, we showed that there exist constants $C, \delta >0$, depending on $U$, such that if $d_{C^0}^{path}(\phi^t_G, \phi^t_H) < \delta$, then 
  \begin{equation} \label{estimate for paths}   
      |c(a,G) - c(a,H)| \leq C \, d_{C^0}^{path}(\phi^t_G, \phi^t_H).
   \end{equation}
   
  It is well known that spectral invariants of a properly normalized Hamiltonian $H$ depend on the homotopy class (rel. endpoints) of $\phi^t_H, \; t \in [0,1]$.  Thus, it is not surprising that the above estimate depends on the $C^0$ distance between the entire paths $\phi^t_G, \; \phi^t_H$ and not just $d_{C^0}(\phi^1_G, \phi^1_H)$.  The main goal of this paper is to show that on monotone symplectic manifolds there exist estimates which are similar to (\ref{estimate for paths}), but depend only on the endpoints $\phi^1_G, \phi^1_H$.  A symplectic manifold is called monotone if $[\omega]|_{\pi_2(M)} = \lambda c_1|_{\pi_2(M)}$ where $c_1$ is the first Chern class of $M$ and $\lambda \in \mathbb{R}$.  We say $M$ is positively monotone if $\lambda > 0$ and negatively monotone if $\lambda < 0$.  Monotone symplectic manifolds are abundant: complex projective spaces are examples of positively monotone symplectic manifolds and examples of negatively monotone manifolds are provided by hypersurfaces of the form $z_0^m + \cdots + z_n^m = 0$ in $\mathbb{C}P^n$, where $m > n+1.$  We can now state the main theorem of this article:
  
  \begin{thm}\label{Main Theorem}
  Let $(M^{2n}, \omega)$ denote a closed, monotone symplectic manifold with $[\omega]|_{\pi_2(M)} =  \lambda c_1|_{\pi_2(M)}$.  Suppose that $H, G \in C^{\infty}_c([0,1] \times (M \setminus U))$.  There exist constants $C, \delta > 0$, depending on $U$, such that if $d_{C^0}(\phi^1_G, \phi^1_H) \leq \delta,$ then
  \begin{enumerate}
  \item $|c(a,G) - c(a,H)| \leq C \, d_{C^0}(\phi^1_G, \phi^1_H)$ when $\lambda \leq 0$,
  \medskip
  \item $|c(a,G) - c(a,H)| \leq C \, d_{C^0}(\phi^1_G, \phi^1_H) + n \lambda$ when $\lambda > 0.$
  \end{enumerate}  
  \end{thm}
  
  \begin{rem}
  \begin{enumerate}
  \item Example 2.4 of \cite{Se11} can easily be modified to prove that the first estimate is sharp in the sense that a locally Lipschitz estimate is optimal.  We do not know if the additive constant $n \lambda$ in the second estimate is necessary; it appears as a byproduct of our methods and it may be possible to remove or improve it.
  \item The assumption that $H$ and $G$ vanish on $U$ is essential.  See Remark \ref{Remark on descent}  below.
  \end{enumerate}
  \end{rem}
  The proof of Theorem \ref{Main Theorem} relies on the $\epsilon$-shift technique which was developed in \cite{Se11}.  
   
   \begin{defn} \label{Epsilon Shift Definition} 
     Fix a positive real number $\epsilon$.  A subset of a symplectic manifold, $U \subset M$, is said to be $\epsilon$-shiftable if there exists a Hamiltonian diffeomorphism, $\phi$, such that 
     $$d(p, \phi(p)) \geq \epsilon \text{  } \forall p \in U.$$
         
  \end{defn}
  
  The main idea of the proof of Theorem \ref{Main Theorem} is as follows: the triangle inequality reduces the theorem to the case where $G=0$ and $a=1$.  We construct an appropriate Morse function $f$ such that $\phi^1_f \;$ $\epsilon$-shifts $M \setminus U$.  The $\epsilon$-shift technique allows us to bound $|c(1,H)|$ by actions of periodic orbits of $f$ with Conley-Zehnder index $2n$.  We then use the monotonicity assumption to show that actions of such orbits satisfy the right estimates.  The details of this argument are carried out in Section \ref{Proof of Main Theorem}.
  
  \subsection{Descent of spectral invariants} \label{Descent of spectral invariants}
  If a Hamiltonian $H$ is mean-normalized, i.e. $\int_{M}H_t \omega^n = 0$ for each $t \in [0,1]$, then the spectral invariant $c(a, H)$ depends only on the homotopy class (rel. endpoints) of $\phi^t_H \; t \in [0,1]$ in $Ham(M)$.  If $H$ is not mean-normalized, but $H \in C^{\infty}_c(M\setminus U)$, then $c(a, H)$ depends on the homotopy class (rel. endpoints) of $\phi^t_H \; (t \in [0,1])$ in $Ham_c(M \setminus U)$. Hence, spectral invariants can be defined on $\widetilde{Ham}(M)$ and $\widetilde{Ham}_c(M\setminus U)$. 
  
   In \cite{Sc00}, Schwarz showed that if $M$ is symplectically aspherical, i.e. $\omega|_{\pi_2} = c_1|_{\pi_2} =0,$ then spectral invariants (of mean-normalized Hamiltonians) descend to $Ham(M)$, i.e. $c(a,H) =  c(a,G)$ if $\phi^1_H = \phi^1_G$ where $H$ and $G$ are assumed to be mean-normalized.  In fact, if $M$ is weakly exact, i.e. $\omega|_{\pi_2} = 0$, Schwarz's argument yields the same conclusion.  Descent of spectral invariants has some significant consequences.  For example, Ostrover \cite{Os03} showed that $Ham(M)$ has infinite Hofer diameter whenever spectral invariants descend.  For many applications it is sufficient that the asymptotic spectral invariants descend.  These are defined as follows: for $H, G \in C^{\infty}([0,1] \times M)$ define $H\#G(t,x) = H(t,x) + G(t, (\phi^t_H)^{-1}(x))$; the flow of $H\#G$ is $\phi^t_H \circ \phi^t_G$.  Following Entov and Polterovich \cite{EP03}  we define the asymptotic spectral invariant by 
  $$\bar{c}(a, H) = \liminf_{n \to \infty} \frac{c(a, H^{\#n})}{n}.$$
  
   Not surprisingly, the search for conditions under which (asymptotic) spectral invariants descend has attracted much attention.  In \cite{EP03}, Entov and Polterovich use the Seidel representation to show that the asymptotic spectral invariants descend in the case of $\mathbb{C}P^n$.  In \cite{Mc10b}, McDuff discusses this topic in depth and, using the Seidel representation, discovers several such conditions.  McDuff's criteria essentially require that many of the genus zero Gromov-Witten invariants of $M$ vanish.  In \cite{Br11}, Branson expands on McDuff's work and proves that the asymptotic spectral invariants descend in many new cases such as monotone products of complex projective spaces and the Grassmanian $G(2,4)$.  
  
  It follows immediately from Theorem \ref{Main Theorem} that spectral invariants descend from $\widetilde{Ham}_c(M\setminus U)$ to $Ham_c(M\setminus U)$ on negatively monotone manifolds and asymptotic spectral invariants descend to $Ham_c(M \setminus U)$ on positively monotone manifolds.
  
  \begin{cor} \label{Cor: descent of spectral invariants}
  Let $(M^{2n}, \omega)$ denote a closed, monotone symplectic manifold with $[\omega]_{\pi_2(M)} =  \lambda c_1|{\pi_2(M)}$.  Suppose that $\phi^1_H = \phi^1_G$ where $H, G \in C^{\infty}_c([0,1] \times (M \setminus U))$.  Then,
  \begin{enumerate}
  \item $c(a,H) = c(a, G)$ when $\lambda \leq 0$,   
  \item $|c(a,H) -c(a,G)| \leq n \lambda$ when $\lambda > 0$. 
  \end{enumerate}
  
  \end{cor}
  
  \begin{rem}\label{Remark on descent}
  The assumption that $H$ and $G$ vanish on $U$ is essential:  in \cite{Os06}, Ostrover showed that the asymptotic spectral invariants of mean-normalized Hamiltonians do not descend in the case of the monotone one point blow up of $\mathbb{C}P^2$; see also Remark 1.4 in \cite{Mc10b}. We discuss this issue further in Remark \ref{Second remark on descent} of Section \ref{Proof of Main Theorem}.

  \end{rem}
  
  \subsection{Applications to Hofer geometry} \label{Applications to Hofer geometry}
  The Hofer length of a Hamiltonian path $\phi^t_H \in PHam(M) \; t\in [0,1]$ is defined to be
  
  $$\mathcal{L}(\phi^t_H) = \Vert H \Vert_{(1,\infty)} := \int_{0}^{1} (\max_{M}H_t - \min_{M} H_t) dt.$$
  The Hofer norm of a Hamiltonian diffeomorphism $\psi \in Ham(M)$ is given by the following expression:
  $$\Vert \psi \Vert_{Hofer} = \inf \{ \Vert H \Vert_{(1,\infty)}: \psi = \phi^1_H \}.$$
  Similarly, if $\psi \in Ham_c(M \setminus U)$ then 
  $$\Vert \psi \Vert_{Hofer} = \inf \{ \Vert H \Vert_{(1,\infty)}: H \in C^{\infty}_c([0,1] \times (M\setminus U)) \; , \psi = \phi^1_H \}.$$
  This norm was introduced by Hofer in \cite{H90}.  Its non-degeneracy was established in \cite{H90} on $\mathbb{R}^{2n}$ and in \cite{LM95} on general symplectic manifolds.  The Hofer distance is given by: $d_{Hofer}(\phi, \psi) = \Vert \phi^{-1} \psi \Vert_{Hofer}.$ For further details on Hofer geometry we refer the reader to \cite{HZ, P}.
  
  The Hofer (pseudo)-norm is also defined on $\widetilde{Ham}(M)$ and $\widetilde{Ham}_c(M \setminus U)$; one takes infimum over all paths in the same homotopy class.  In this case it is not known if the Hofer (pseudo)-norm defines a norm.  The difficulty here is that there may exist non-contractible Hamiltonian loops  with zero Hofer length.
  
  \subsubsection{Infiniteness of Hofer diameter}
  It is believed, but not proven, that the Hofer norm is always unbounded on $Ham(M)$ and $Ham_c(M\setminus U)$. If it is unbounded on any of these groups we will say that the group has \emph{infinite (Hofer) diameter}. Infiniteness of diameter has been established on many closed manifolds.  An argument due to Ostrover \cite{Os03} proves that $\widetilde{Ham}(M)$ has infinite diameter.  Whenever spectral invariants or their asymptotic versions descend, Ostrover's argument yields the same consequence for $Ham(M)$.  However, his argument does not apply to  $Ham_c(M \setminus U)$: the argument relies on spectral invariants of mean-normalized Hamiltonians descending and thus it can not be combined with Corollary \ref{Cor: descent of spectral invariants}.
  
  For a non-closed manifold, such as $M \setminus U$, infiniteness of Hofer diameter of $\widetilde{Ham}_c(M \setminus U)$ can be established using the Calabi homomorphism: 
  $$Cal: \widetilde{Ham}_c(M \setminus U) \rightarrow \mathbb{R}, \;\;\; Cal(\tilde{\phi^t}):= \int_{0}^{1} \int_{M} H(t, \cdot) \omega^n dt,$$
  
  \noindent where $H \in C^{\infty}_c([0,1] \times (M\setminus U))$ is any Hamiltonian generating $\tilde{\phi^t}$.  $Cal$ is a homomorphism and it can easily be seen that 
  $Cal(\tilde{\phi^t}) \leq Vol(M) \Vert \tilde{\phi^t} \Vert_{Hofer}.$  Hence, we see that $\widetilde{Ham}_c(M \setminus U)$ always has infinite  diameter and if $Cal$ descends to $Ham_c(M \setminus U)$, then it has infinite diameter as well.  
  If $Cal$ does descend then the interesting question is whether the kernel of $Cal$ has infinite Hofer diameter.  In a sense, as noted in \cite{Mc10a}, this kernel plays the role of $Ham(M)$ for $M\setminus U$: Banyaga \cite{Ba78} showed that $\ker(Cal)$ is a perfect group and hence it admits no non-trivial homomorphism to $\mathbb{R}$.   As McDuff points out in \cite{Mc10a} (see Remark 3.11) if $Cal$ does not descend there are no standard methods for deciding whether $Ham_c(M \setminus U)$ has infinite diameter; the same is true of $kernel$ of $Cal$ if it does descend.
  
  One case where $\ker(Cal)$ is known to have infinite diameter is that of $Ham_c(B^{2n})$; Biran, Entov, and Polterovich \cite{BEP04} have shown that $Ham_c(B^{2n})$ admits more than one Calabi quasimorphism and that is sufficient for proving that kernel of $Cal$ has infinite diameter.  
 
  In Section \ref{Proofs of applications to Hofer geometry} we will use Corollary \ref{Cor: descent of spectral invariants} to settle new cases of the above questions.  Recall that $V \subset M$ is said to be displaceable if there exists $\psi \in Ham(M)$ such that $\psi(V) \cap V = \emptyset$.  More generally, $V$ is said to be stably displaceable if $V \times S^1 \subset M \times T^*S^1$ is displaceable.  Stable displaceability does not imply displaceability as shown in Example 1.28 of \cite{EP09a}.
  
  \begin{thm} \label{Th: infiniteness of diameter}
  Suppose that $M$ is monotone and let $U$ denote an open subset of $M$ whose closure, $cl(U)$, is stably displaceable.  Then, $Ham_c(M\setminus U)$ has infinite Hofer diameter.  Furthermore, if the Calabi homomorphism, $Cal,$ descends to $Ham_c(M  \setminus U),$ then $\ker(Cal)$ has infinite Hofer diameter.  
  \end{thm} 
  The above theorem can be viewed as generalization of the same facts about $Ham_c(B^{2n})$; see Remark \ref{Strom's Remark}.
  
 \subsubsection{Hofer's norm v.s. $C^0$ norm: Le Roux's question}
 The relation between Hofer's and the $C^0$ norm is mysterious.  The $C^0$ norm is never continuous with respect to Hofer's norm; any two points can be mapped to one another with arbitrarily small energy.  In \cite{H93}, Hofer compares the $C^0$-distance and the Hofer distance on $Ham_c(\mathbb{R}^{2n})$ and obtains the well known $C^0$-Energy estimate:
  $$ d_{Hofer}(\phi, \psi) \leq 256\; d_{C^0}(\phi, \psi).$$
  
  No estimate of this kind holds on compact manifolds.  In fact, one can show that on any surface there exists a sequence of Hamiltonian diffeomorphisms which converges to the identity in $C^0$-topology, but diverges with respect to Hofer's metric.  Attempting to understand the relation between these two distances led Le Roux \cite{Le10} to pose the following problem.  Let $X$ denote a compact manifold possibly with boundary.  Let $\mathcal{E}_{A}(X)$ denote the complement of the ball of radius $A$, in Hofer's metric, centered at the identity:
  
  $$\mathcal{E}_{A}(X) = \{ \phi \in Ham_c(X): \Vert \phi \Vert_{Hofer} > A\}.$$

  \begin{que} \label{Le Rroux's question}
  Does $\mathcal{E}_{A}(X)$ have non-empty $C^0$ interior for any $A>0$?
  \end{que}
    Le Roux's original question was posed for $X = B^2$, the unit ball in $\mathbb{R}^2$.  In \cite{EPP09}, Entov, Polterovich, and Py answered Le Roux's question affirmatively for $X=B^{2n}$, the unit ball in $\mathbb{R}^{2n}$.  As pointed out in \cite{EPP09} this question makes sense for any manifold $X$.  In Section  \ref{Proofs of applications to Hofer geometry}, we will prove the following:
    
    \begin{thm}\label{Th: Le Roux's question}
    Suppose that $M$ is monotone and let $U$ denote an open subset of $M$ whose closure, $cl(U)$, is stably displaceable. Then, for any $A  >0$, $\mathcal{E}_{A}(M \setminus U)$ has non-empty $C^0$ interior.
    \end{thm}
    The following observation is due to Strom Borman:
    \begin{rem}\label{Strom's Remark}
    Theorems \ref{Th: infiniteness of diameter} and \ref{Th: Le Roux's question} can be viewed as generalizations of the same facts about $Ham_c(B^{2n}),$ by taking $M$ to be $\mathbb{C}P^n$ and $U$ a small neighborhood of $\mathbb{C}P^{n-1}$.  $U$ is not displaceable because the intersection product of $\mathbb{C}P^{n-1}$ with itself is non-trivial. However, $U$ is stably displaceable; see Corollary 11 in \cite{Bor10}.
    \end{rem}
  
  \subsection{Extension to compact manifolds with convex boundary}\label{Non-closed manifolds}
  
    Theorems \ref{Main Theorem}, \ref{Th: infiniteness of diameter}, \ref{Th: Le Roux's question}, and Corollary \ref{Cor: descent of spectral invariants} extend to some compact manifolds with boundary, e.g. $T^*_rN$ the cotangent ball bundle of radius $r$ over a closed manifold $N$.  See Section \ref{Extension} for more details.

  \section{Review of spectral invariants} \label{Spectral Invariants}
  In this section we briefly review the construction of spectral invariants on closed symplectic manifolds.  For further details we refer the interested reader to \cite{MS2, Oh06a}.  
  
   Define 
   $$\Gamma:= \frac{\pi_2(M)}{\ker(c_1) \cap \ker([\omega])}.$$
   The Novikov ring of $(M, \omega)$ is defined to be
   $$ \displaystyle \Lambda = \{ \sum_{A\in \Gamma}{a_A A }: a_A \in \mathbb{Q}, (\forall C \in \mathbb{R}) (|\{A: a_A \neq 0, \int_{A} {\omega} < C\}| < \infty)\}.$$  
   Let $\Omega_0(M)$ denote the space of contractible loops in $M$.  $\Gamma$ forms the group of deck transformations of a covering $\tilde{\Omega}_0(M) \rightarrow \Omega_0(M)$ called the Novikov covering of $\Omega_0(M)$ which can be described as follows:
   $$\tilde{\Omega}_0(M) = \frac{ \{ [z,u]: z \in \Omega_0(M) , u: D^2 \rightarrow M , u|_{\partial D^2} = z \}}{[z,u] = [z', u'] \text { if } z=z' \text{ and } \bar{u} \# u' = 0 \text{ in } \Gamma},$$
   where $\bar{u} \# u'$ denotes the sphere obtained by gluing $u$ and $u'$ along their common boundary with the orientation on $u$ reversed.  
   
   The action functional, associated to a Hamiltonian $H \in C^{\infty}([0,1]\times M)$, is the map $\mathcal{A}_H:\tilde{\Omega}_0(M) \rightarrow \mathbb{R}$ given by
   $$\mathcal{A}_H([z,u]) =  \int H(t,z(t))dt \text{ }- \int_{u} \omega.$$
   Note that $$\mathcal{A}_H([z,u\#A]) = \mathcal{A}_H([z,u]) - \omega(A),$$ for every $A \in \Gamma$. 
   $Crit(\mathcal{A}_H) = \{ [z,u] : \text{ z is a 1-periodic orbit of } X_H \}$ denotes the set of critical points of $\mathcal{A}_H$.  The action spectrum of $H$ is defined to be the set of critical values of the action functional, i.e., $Spec(H) = \mathcal{A}_H(Crit(\mathcal{A}_H))$.  $Spec(H)$ is a measure zero subset of $\mathbb{R}$.
   
    We say that a Hamiltonian $H$ is non-degenerate if the graph of $\phi^1_H$ intersects the diagonal in $M \times M$ transversally.  The Floer chain complex of (non-degenerate) $H$, $CF_*(H)$, is generated as a module over $\Lambda$ by $Crit(\mathcal{A}_H)$.  The complex $CF_*(H)$ is graded by the Conley-Zehnder index, $\mu_{cz}: Crit(\mathcal{A}_H) \rightarrow \mathbb{Z}$, which satisfies
    $$\mu_{cz}([z,u\#A]) = \mu_{cz}([z,u]) - 2 c_1(A),$$
    \noindent for every $A \in \Gamma$.   Various conventions are used for defining the Conley-Zehnder index.  We fix our convention as follows: let $f$ denote a $C^2$-small Morse function.  For every critical point $p$ of $f$, we require that 
   $$i_{Morse}(p) = \mu_{cz}([p, u_p]),$$
   where $u_p$ is a trivial capping disc and $i_{Morse}(p)$ is the Morse index of $p$.  The boundary map of this complex is obtained, formally, by counting isolated negative gradient flow lines of $\mathcal{A}_H$.  The homology of this complex, $HF_*(H)$, is naturally isomorphic to $QH^*(M) = H^*(M) \otimes \Lambda$, the quantum cohomology of $M$.  We denote this natural isomorphism, which is called the PSS isomorphism \cite{PSS}, by 
   $\Phi_{pss} : QH^*(M) \rightarrow HF_*(H)$.  Our conventions imply that $\Phi_{pss}$ identifies $QH^k(M)$ with $HF_{2n-k}(H)$.
   
   Given $\displaystyle \alpha = \sum_{[z,u] \in Crit(\mathcal{A}_H)}{a_{[z,u]}[z,u]} \in CF_*(H)$ we define the action level of $\alpha$ by
   $$\lambda_H(\alpha) = \max\{\mathcal{A}_H([z,u]): a_{[z,u]} \neq  0 \}.$$
   Finally, given a non-zero quantum cohomology class $a$, we define the spectral invariant associated to $H$ and $a$ by
   $$c(a,H) = \inf \{ \lambda_H(\alpha): [\alpha] = \Phi_{pss}(a)\},$$
   where $[\alpha]$ denotes the Floer homology class of $\alpha$.  It was shown in \cite{Oh05b} that $c(a,H)$ is well defined, i.e., it is independent of the auxiliary data (almost complex structure) used to define it and $c(a,H) \neq -\infty$.
   
   Thus far we have defined $c(a,H)$ for non-degenerate $H$.  The spectral invariants of two non-degenerate Hamiltonians $H$, $G$ satisfy the following estimate $$|c(a, H) - c(a, G)| \leq \int_{0}^{1} \max_{x\in M} |H_t -G_t|  dt.$$ This estimate allows us to extend $c(a, \cdot)$ continuously to all smooth (in fact continuous) Hamiltonians.  
   
   We will now list, without proof, some properties of $c$ which will be used later on.  Recall that the composition of two Hamiltonian flows, $\phi^t_H \circ \phi^t_G$, and the inverse of a flow, $(\phi^t_H)^{-1},$ are Hamiltonian flows generated by $H\#G(t,x) = H(t,x) + G(t, (\phi^t_H)^{-1}(x))$ and $\bar{H}(t,x) = -H(t, \phi^t_H(x))$, respectively. 
   
   \begin{prop} \label{Properties of Spectral Invariants}(\cite{Oh05b, Oh06a, Sc00, Us10a})\\
    The function $c: (QH^*(M) \setminus {0}) \times C^{\infty}([0,1] \times M)  \rightarrow \mathbb{R}$ has the following properties:
    \begin{enumerate}
    \item (Shift)If $r : [0,1] \rightarrow \mathbb{R}$ is smooth, then $c(a,H+r) = c(a,H) + \int_{0}^{1}{r(t) dt}.$
    \item (Normalization) $c(1,0) =0$.
    \item (Symplectic Invariance) $c(\eta^*a, \eta^*H ) = c(a, H)$ for any symplectomorphism $\eta$.
    \item (Monotonicity) If $H \leq G$, then $c(a,H) \leq c(a,G)$.
    \item (Triangle Inequality) $c(a*b,H\#G) \leq c(a,H) + c(b,G)$ where $*$ denotes the quantum product in $QH^*(M)$.
    \item ($L^{(1,\infty)}-continuity$) If $H$ and $G$ are mean-normalized, or if $H, G \in C^{\infty}_c([0,1]\times (M\setminus U),$ then $|c(a,H) - c(a,G)| \leq \Vert H - G \Vert_{(1, \infty)}$.
    \item (Spectrality) $c(a,H) \in Spec(H)$ for non-degenerate $H$.
    \item (Homotopy Invariance) Suppose that $H$ and $G$ are mean-normalized and generate the same element of $\widetilde{Ham}(M)$.  Then, $c(a,H) = c(a,G)$.  The same conclusion holds if $H, G \in C^{\infty}_c([0,1]\times (M\setminus U))$ generate the same element of $\widetilde{Ham}_c(M\setminus U)$.
    
    \end{enumerate}   
   \end{prop}
   
    The homotopy invariance property for Hamiltonians in $C^{\infty}_c([0,1]\times (M\setminus U))$ follows from the same property for mean-normalized Hamiltonians: if $H, G \in C^{\infty}_c([0,1]\times (M\setminus U))$ generate the same element of $\widetilde{Ham}_c(M\setminus U)$ then $\int_{0}^{1}(\int_M{H_t \omega^n})\;dt = \int_{0}^{1}(\int_M{G_t\omega^n}) \;dt$.  The rest of the above properties are standard.
   
  \section{Proof of Theorem \ref{Main Theorem}} \label{Proof of Main Theorem}
  In this section we prove Theorem \ref{Main Theorem}. 
  
  \begin{rem}\label{Second remark on descent}    
 As mentioned in the introduction, Corollary \ref{Cor: descent of spectral invariants} does not imply that (asymptotic) spectral invariants of mean-normalized Hamiltonians descend.  Suppose that $H \in C^{\infty}_c([0,1] \times M)$ is mean-normalized and that it generates a loop in $Ham_c(M\setminus U)$.  Then, it is not hard to see that Corollary \ref{Cor: descent of spectral invariants} implies that 
 \begin{itemize}
 \item $c(1,H) +  \frac{1}{Vol(M)} \;Cal (\phi^t_H) = 0$, if $\lambda \leq 0$, and 
 \item $|c(1,H) + \frac{1}{Vol(M)} \; Cal(\phi^t_H)| \leq n \lambda$, if $\lambda > 0.$
 \end{itemize}

 Thus, asymptotic spectral invariants of mean-normalized Hamiltonians that generate paths in $Ham_c(M\setminus U)$ descend if and only if the Calabi homomorphism descends to $Ham_c(M\setminus U)$.   In Remark 3.10 of \cite{Mc10a}, McDuff gives a prescription for construction loops in $Ham_c(M\setminus U)$ with non-vanishing Calabi invariant.
\end{rem}
  
 We will be using the following terminology which we are borrowing from \cite{Sch06, Us10a}.
  
  \begin{defn}
  A time independent Hamiltonian $f: M \rightarrow \mathbb{R}$ is said to be \emph{slow}, if its Hamiltonian flow $\phi^t_f$ has no non-trivial, contractible periodic orbits of period at most $1$.  
  \end{defn}
  
  Recall that we say $U \subset M$ is $\epsilon$-shiftable if there exists a Hamiltonian diffeomorphism, $\phi$, such that 
     $d(p, \phi(p)) \geq \epsilon \text{  } \forall p \in U$; see Definition \ref{Epsilon Shift Definition}.  The following theorem, which constitutes the main step towards the proof of Theorem \ref{Main Theorem}, is the main reason for introducing the notion of $\epsilon$-shiftability.  
  
  \begin{thm} \label{Epsilon Shift Theorem}
  
 Suppose that the support of a Hamiltonian $H$ can be $\epsilon$-shifted by $\phi^1_f \in Ham(M)$, where $f$ denotes a slow Hamiltonian.  If $d_{C^0}(Id, \phi^1_H) < \epsilon$, then
   \begin{enumerate}
    \item $|c(1,H) |< 2 \Vert f \Vert_{\infty}, \text{ if } \lambda \leq 0,$
    \item $|c(1,H) |< 2 \Vert f \Vert_{\infty} + n \lambda, \text{ if } \lambda > 0.$  
   \end{enumerate}
  \end{thm}
  
  We will now present a proof of Theorem \ref{Main Theorem}.  This stage of our proof closely parallels the arguments from a similar stage of \cite{Se11}. 
  
  \begin{proof} [Proof of Theorem \ref{Main Theorem}]
   First, suppose that $G=0, a=1$. We have to show that there exist constants $C, \delta > 0$ such that whenever $d_{C^0}(Id, \phi^1_H) < \delta$ then 
      \begin{itemize} 
     \item $|c(1,H)| \leq C \; d_{C^0}(Id, \phi^1_H), \text{ if } \lambda \leq 0,$
     \item $|c(1,H)| \leq C \; d_{C^0}(Id, \phi^1_H) + n \lambda, \text{ if } \lambda  > 0.$
     \end{itemize}
     
     Pick a slow Morse function $f$ all of whose critical points are contained in $U$, and denote by $X_f$ the Hamiltonian vector field of $f$.  Let $C_1 : = \inf \{\Vert X_f(x)\Vert: x \in M \setminus U\}$.  The set $M \setminus U$ is compact; thus $C_1 >0$ and we can find a sufficiently small $r > 0$, such that for each $s\in [0, r]$ the Hamiltonian diffeomorphism $\phi^s_f\text{ } \frac{C_1 s}{2}$-shifts the set $M\setminus U$.  Now consider $H \in C^{\infty}_c([0,1] \times  ( M \setminus U ))$ such that $d_{C^0}(Id, \phi^1_H) < \frac{C_1 r}{2}$.  By construction, for any $s \in (\frac{2}{C_1} d_{C^0}(Id, \phi^1_H), r]$, the Hamiltonian diffeomorphism $\phi^s_f\text{ } \frac{C_1s}{2}$-shifts the support of $H$, and  $d_{C^0}(Id, \phi^1_H) < \frac{C_1s}{2}$. 
   Hence, we can apply Theorem \ref{Epsilon Shift Theorem} and conclude that 
  \begin{itemize}
    \item $|c(1,H) |< 2 \Vert sf \Vert_{\infty}, \text{ if } \lambda \leq 0,$
    \item $|c(1,H) |< 2 \Vert sf \Vert_{\infty} + n \lambda, \text{ if } \lambda > 0.$  
  \end{itemize} 
   The above inequalities hold for all $s \in (\frac{2}{C_1} d_{C^0}(Id, \phi^1_H), r]$.  Therefore,
   
  \begin{itemize}
    \item $|c(1,H) |< 2 \Vert \frac{2}{C_1} d_{C^0}(Id, \phi^1_H) f \Vert_{\infty} , \text{ if } \lambda \leq 0,$
    \item $|c(1,H) |< 2 \Vert \frac{2}{C_1} d_{C^0}(Id, \phi^1_H) f \Vert_{\infty} + n \lambda, \text{ if } \lambda > 0.$  
  \end{itemize} 
  The result follows, with $C:= 2 \frac{2}{C_1} \Vert f \Vert_{\infty} $ and $\delta := \frac{C_1 r}{2}$.
  
  Now, we consider general $G\in C^{\infty}_c([0,1] \times  (M \setminus U))$ and $a \in QH^*(M) - \{0\}$.  Recall that $d_{C^0}(\phi^1_G, \phi^1_H) = d_{C^0}(Id, \phi^{-1}_G \phi^1_H)$.  From the above we conclude that if $d_{C^0}(\phi^1_G, \phi^1_H) < \delta$ then   
  \begin{itemize}
    \item $c(1,\bar{G} \# H) < C \; d_{C^0}(\phi^1_G, \phi^1_H) , \text{ if } \lambda \leq 0,$
    \item $c(1,\bar{G} \# H) < C \; d_{C^0}(\phi^1_G, \phi^1_H) + n \lambda, \text{ if } \lambda > 0.$  
  \end{itemize} 
  
  From the triangle inequality for spectral invariants we get that $c(a,H) - c(a, G) \leq c(1,\bar{G} \# H)$, which combined with the above inequalities gives us 
  \begin{itemize}
    \item $c(a,H) - c(a, G) < C \; d_{C^0}(\phi^1_G, \phi^1_H) , \text{ if } \lambda \leq 0,$
    \item $c(a,H) - c(a, G) < C \; d_{C^0}(\phi^1_G, \phi^1_H) + n \lambda, \text{ if } \lambda > 0.$  
  \end{itemize} 
     
  Similarly, we get the same inequalities for $c(a,G) - c(a, H)$.  The result then follows.
  \end{proof}
  
  We will now provide a proof for Theorem \ref{Epsilon Shift Theorem}.    
  \begin{proof} [Proof of Theorem \ref{Epsilon Shift Theorem}]
  
  We may assume, by slightly $C^{\infty}$-perturbing $f$, that it is Morse.  Let $Crit(f)$ and $Fix(\phi^1_f)$ denote the set of critical points of $f$ and fixed points of $\phi^1_f$, respectively.  Since $f$ is slow we have $Crit(f)  = Fix(\phi^1_f)$.  Let $supp(H)$ denote the support of $H$ and set $B = M \setminus supp(H)$.  The assumption that $supp(H)$ is $\epsilon$-shifted by $\phi^1_f$ implies that $Fix(\phi^1_f) \subset B$.
  
  We claim that $Fix(\phi^1_H \circ \phi^1_f) = Fix(\phi^1_f)$.  Clearly,  $Fix(\phi^1_f) \subset Fix(\phi^1_H \circ \phi^1_f)$ because  $Fix(\phi^1_f) \subset B$.  For the other containment suppose that $p \in Fix(\phi^1_H \circ \phi^1_f)$.  First, note that if $p \in supp(H)$ then we would have $d(p, \phi^1_f(p)) \geq \epsilon > d_{C^0}(Id, \phi^1_H)$ and so $\phi^1_H$ can not move $\phi^1_f(p)$ back to $p$.  Thus, we see that $p \notin supp(H)$.  Next, observe that $\phi^1_f(p) \notin supp(H)$: if $\phi^1_f(p)$ were in $supp(H)$ then we would get $p = \phi^1_H \phi^1_f(p) \in Supp(H)$.  It then follow that $p = \phi^1_H \phi^1_f(p) = \phi^1_f(p)$, and thus  $Fix(\phi^1_H \circ \phi^1_f) \subset Fix(\phi^1_f)$.  
  
  The above implies that $$Spec(H \# f)
   = \{ \mathcal{A}_{H \# f}([p, A]): p \in Crit(f), \; A \in \Gamma \}$$ $$ = \{ f(p) - \omega(A) : p \in Crit(f), \;  A \in \Gamma \},$$
  where the second equality follows from the fact that $Crit(f) \subset B$.  Now, observe that the Hamiltonian $\phi^1_H\phi^1_f$ is non-degenerate because it coincides with $\phi^1_f$ on a neighborhood of its fixed points, and so $c(1, H \# f)$ is attained by a periodic orbit of Conley-Zehnder index $2n$.  Here, we have used the fact that $\Phi_{pss}$ is an isomorphism between $QH^0(M)$ and $HF_{2n}(H \# f)$.  
  
  It follows from the above that there exist $p \in Crit(f)$ and $A \in \Gamma$ such that $$c(1, H \# f) = f(p) - \omega(A) \; \; \text{and } \mu_{cz}([p, A]) = 2n.$$  Recall, from Section \ref{Spectral Invariants}, that $\mu_{cz}([p, A]) = i_{Morse}(p) - 2c_1(A)$ and therefore,
  $c_1(A) = \frac{i_{Morse}(p) - 2n}{2}$.  Since $0 \leq i_{Morse}(p) \leq 2n$ we conclude that 
  \begin{equation} \label{bound on c_1}
     -n \leq c_1(A) \leq 0.
  \end{equation}
  
  We finish the proof by considering the following two cases:
  
  \noindent \textbf{Case 1: $\lambda \leq 0$.} We have $c(1, H \# f) = f(p) - \omega(A) = f(p) -  \lambda c_1(A) \leq f(p) \leq \Vert f \Vert_{\infty}$; here we have used the right hand side of (\ref{bound on c_1}).  Using the triangle inequality, we get $c(1, H ) \leq c(1, H \# f) + c(1, -f) \leq 2 \Vert f \Vert_{\infty}.$  Now, note that we can repeat all of the above for the Hamiltonian $\bar{H}$ and so we $c(1, \bar{H}) \leq 2 \Vert f \Vert_{\infty}.$  Finally, using the fact that $ 0 \leq c(1, \bar{H}) + c(1, H)$, we obtain $$|c(1, H)| \leq 2 \Vert f \Vert_{\infty}.$$
  
  \noindent \textbf{Case 2: $\lambda > 0$.}  $c(1, H \# f) = f(p) - \omega(A) = f(p) -  \lambda c_1(A) \leq f(p) + n \lambda \leq \Vert f \Vert_{\infty} + n \lambda$; here we have used the left hand side of (\ref{bound on c_1}).  Now, using the argument from Case 1 we obtain:   $$|c(1, H)| \leq 2 \Vert f \Vert_{\infty} + n \lambda.$$
  \end{proof}

  \section{Proofs of applications to Hofer geometry} \label{Proofs of applications to Hofer geometry}
  In this section we prove Theorems  \ref{Th: infiniteness of diameter} and \ref{Th: Le Roux's question}. 
  
  The statements of Theorems \ref{Th: infiniteness of diameter} and \ref{Th: Le Roux's question} require the closure of $U$ to be stably displaceable.  However, it can be extracted from the following proofs that the above theorems hold under less restrictive conditions on $U$; it is sufficient to require that $cl(U)$ is contained in an open set $V$ with the property that $c(1, \cdot)$ is bounded on $C^{\infty}_c(V)$. 
   The following lemma, which is due to Borman, will be needed in our proof:
   \begin{lem}\label{Strom's Lemma}(Borman \cite{Bor10})
   Let $V \subset M$ denote a stably displaceable set, and suppose that $F \in C^{\infty}_c([0,1]\times V)$.  Then, there exists a constant $E$, depending on $V$, such that $|c(1, F)| \leq E$.
   \end{lem}
   We will provide a proof for this lemma at the end of this section.
  
  \begin{proof}[Proof of Theorem \ref{Th: infiniteness of diameter}] 
 First, we show that $Ham_c(M\setminus U)$ has infinite diameter. By Corollary \ref{Cor: descent of spectral invariants} it is sufficient to show that there exists $H \in C^{\infty}_c([0,1] \times (M \setminus U))$ with arbitrarily large $c(1,H)$: indeed Corollary \ref{Cor: descent of spectral invariants}, combined with $L^{(1, \infty)}$ continuity of spectral invariants, implies that 
  \begin{itemize}
  \item $c(1,H) \leq \Vert \phi^1_H \Vert_{Hofer}$ if $\lambda \leq 0$, and
  \item $c(1,H) - n \lambda \leq \Vert \phi^1_H \Vert_{Hofer}$ if $\lambda > 0$.
  \end{itemize}
  
   $cl(U)$ is compact and stably displaceable, and thus there exists a stably displaceable open set $V$ which contains $cl(U)$.  Let $F$ denote an autonomous Hamiltonian with the following properties:
  
  \begin{enumerate}
  \item $supp(F) \subset V,$
  \item $F(p) = -C$ for all $p \in U$, where $C$ denotes a large positive number. 
  \end{enumerate}
  
  It follows from Lemma \ref{Strom's Lemma} that $|c(1,F)| \leq E$, for some constant $E$.  Let $H = F + C$; note that $H \in C^{\infty}_c([0,1] \times (M \setminus U))$.  By the shift property of spectral invariants we have $$c(1,H) = c(1, F) + C \geq C - E.$$  It follows that $Ham_c(M \setminus U)$ has infinite diameter. 
  
   Now, suppose the $Cal$ descends to $Ham_c(M \setminus U)$.  By modifying $F$ on $V \setminus U$ we can ensure that $\int_{M}{F \omega ^n}  = -C \; Vol(M), $ where $ Vol(M)$ denotes the volume of $M$.  It then follows that 
   $$Cal(\phi^1_H) = \int_{M}{F+C \; \omega^n} + \int_{M} F \omega^n + C\; Vol(M) = 0.$$
  \end{proof}
  
  \begin{proof}[Proof of Theorem \ref{Th: Le Roux's question}]
  It follows from the above proof of Theorem \ref{Th: infiniteness of diameter} that there exists $H \in C^{\infty}_c([0,1]\times (M \setminus U))$ such that $c(1,H) >> A $. By Theorem \ref{Main Theorem} there exists a small $\delta \in \mathbb{R}$ such that
  if $d_{C^0}(\phi^1_G, \phi^1_H) < \delta$ then 
  \begin{itemize}
  \item $c(1,H) - \delta \leq c(1,G)$ if $\lambda \leq 0$, and 
  \item $c(1,H) -\delta - n \lambda \leq c(1,G)$ if $\lambda >0$.
  \end{itemize}
  
   \noindent By Corollary \ref{Cor: descent of spectral invariants},  we have 
  \begin{itemize}
  \item  $c(1,G) \leq \Vert \phi^1_G \Vert_{Hofer}$ if $\lambda \leq 0$, and
  \item  $c(1,G) - n \lambda \leq \Vert \phi^1_G \Vert_{Hofer}$ if $\lambda > 0$. 
  \end{itemize}
  Combining the above inequalities we get:
  \begin{itemize}
  \item $c(1,H) - \delta \leq \Vert \phi^1_G \Vert_{Hofer} $ if $\lambda \leq 0$, and 
  \item $c(1,H) -\delta - 2n \lambda \leq \Vert \phi^1_G \Vert_{Hofer}$ if $\lambda >0$,
  \end{itemize}
 
  \noindent Because $c(1,H) >> A$, it follows that the $C^0$ open ball of radius $\delta$ centered at $\phi^1_H$ is contained in $\mathcal{E}_A(M \setminus U)$.
   \end{proof}  
   We end this section by proving Lemma \ref{Strom's Lemma}.  We thank Strom Borman for showing us the proof of this lemma.  The argument presented here, closely follows the proof of Theorem 2 from \cite{Bor10}.  
   \begin{proof}[Proof of \ref{Strom's Lemma}]
   As argued in \cite{Bor10}, we may assume that $V \times S^1$ is displaceable in $M \times S^2$, where $S^2$ is a round sphere in $\mathbb{R}^3$ with coordinates $(x,y,z)$ and equipped with the induced area form, and $S^1$ is the equatorial circle given by $z=0$.  Denote by $\pi_1: M \times S^2 \rightarrow M$ the standard projection.  In \cite{Bor10}, Borman constructs an open covering of $M \times S^2$, $\{\mathcal{U}_0, \mathcal{U}_1, \mathcal{U}_2\}$, which admits a subordinate partition of unity $\{\phi_0, \phi_1, \phi_2\}$ with the property that $\phi_i \# \phi_j = \phi_i + \phi_j$, and 
   \begin{equation} \label{eq: Poisson Commutation}
   \pi_1^*F \cdot \phi_i \; \# \; \pi_1^*F \cdot \phi_j = \pi_1^*F \cdot \phi_i +  \pi_1^*F \cdot \phi_j. 
   \end{equation}
   Furthermore, Borman's construction ensures that $\mathcal{U}_1$ and  $\mathcal{U}_2$ are displaceable, and although $\mathcal{U}_0$ is not displaceable, the support of $\pi_1^*F \cdot \phi_0$ is contained in a displaceable neighborhood of $V \times S^1$.  Hence, we see that the functions $\pi_1^*F \cdot \phi_i$ have displaceable supports. 
   
      Denote by $1_{\mathrm{M}} \in QH^0(M)$ and $1_{\mathrm{M \times S^2}} \in QH^0(M\times S^2)$ the identity cohomology classes.  Now, by Theorem 5.1 of \cite{EP09a} we have: $c(1_{\mathrm{M}}, F)   =  c(1_{\mathrm{M \times S^2}}, \pi_1^* F).$ Also, Equation (\ref{eq: Poisson Commutation}) implies
      $$c(1_{\mathrm{M \times S^2}}, \pi_1^* F) = c(1_{\mathrm{M \times S^2}}, \Sigma_{i=0}^2 \pi_1^*F \cdot \phi_i)$$
      $$= c(1_{\mathrm{M \times S^2}}, \pi_1^*F  \cdot \phi_0 \; \# \;\pi_1^*F \cdot \phi_1 \; \# \;\pi_1^*F \cdot \phi_2) \leq \Sigma_{i=0}^2 c(1_{\mathrm{M \times S^2}}, \pi_1^*F \cdot \phi_i).$$
          It follows from a well known argument, due to Ostrover \cite{Os03}, that $c(1_{\mathrm{M \times S^2}}, \pi_1^*F \cdot \phi_i) \leq e_i,$ where $e_i$ denote the displacement energy of the support of $\pi_1^*F \cdot \phi_i$.  Hence, it follows that $c(1_{\mathrm{M}}, F) \leq E$,  where $E = \Sigma_{i=0}^2 e_i.$  
       Clearly, the above argument implies that $c(1_{\mathrm{M}}, \bar{F}) \leq E$ as well.  Since $c(1_{\mathrm{M}}, \bar{F}) + c(1_{\mathrm{M}}, F) \geq 0$ we conclude that $|c(1_{\mathrm{M}}, F)| \leq E.$
   \end{proof}
  
\section{Extension to compact manifolds with boundary} \label{Extension}
Let $(X, \omega)$ denote a compact symplectic manifold with boundary.  We denote by $C^{\infty}_c([0,1] \times X)$ the set of Hamiltonians which vanish near $\partial X$, the boundary of $X$.    
    
    It has been shown by Frauenfelder and Schlenk \cite{FS07} that spectral invariants can be defined for $H \in C^{\infty}_c([0,1] \times X)$ if $(X, \omega)$ satisfies certain technical conditions: $\partial X$ must be convex and $\omega$ must satisfy a semi-positivity condition, see \cite{FS07, La11} for details.  $\partial X$ is said to be convex if there exists an outward pointing vector field $V$ along $\partial X$, called the Liouville vector field, such that $\mathcal{L}_V \omega= \omega$ near $\partial X$.  With regards to the semi-positivity condition we only mention that if $[\omega]|_{\pi_2} = \lambda c_1|_{\pi_2}$ for some $\lambda \geq 0$ then the semi-positivity condition required in \cite{FS07, La11} are satisfied, but not if $\lambda < 0$.  Examples of manifolds satisfying all the required technical conditions of \cite{FS07, La11} include $X = T^*_rN$, the cotangent ball bundle of radius $r$ over a closed manifold $N$, and Stein domains.  
    
   Suppose that $X^{2n}$ is a manifold with boundary satisfying all the technical conditions needed to define spectral invariants.  Let $c(a,H)$ denote the spectral invariant associated to $H \in C^{\infty}_c([0,1] \times X)$ and the quantum cohomology class $a$.  (As pointed out by Lanzat in \cite{La11}, there are two sets of possibilities for $a$; it could be an absolute or a relative (to boundary) quantum cohomology class.  Hence, one obtains two sets of spectral invariants for each Hamiltonian).  These spectral invariants satisfy all the standard properties listed in Section \ref{Spectral Invariants}.   
   
   \subsection{Theorem \ref{Main Theorem} on compact manifolds with boundary} 
   Suppose that $[\omega]|_{\pi_2} = \lambda c_1|_{\pi_2}$.  If $\lambda = 0$ then spectral invariants descend to $Ham_c(X)$; see \cite{FS07}.  If $\lambda >0$ then spectral invariants are defined on the universal cover of $Ham_c(X)$.  Our proof of Theorem \ref{Main Theorem} carries over to establish the following:
   
   \begin{thm} \label{Main Theorem non-compact}
   There exist constants $C, \delta > 0$, depending on $X$, such that for any quantum cohomology class $a$ and any Hamiltonians $H, G \in C^{\infty}_c([0,1]\times X)$ if $d_{C^0}(\phi^1_H, \phi^1_G) \leq \delta$, then
   \begin{equation}\label{Main Theorem for Non-closed Manifolds}
   |c(a,G) - c(a,H)| \leq C \, d_{C^0}(\phi^1_G, \phi^1_H) + n \lambda.
   \end{equation}
   \end{thm}
   
   \noindent It then follows that spectral invariants descend ``up to a constant" on positively monotone $X$ and hence asymptotic spectral invariants always descend.  Observe that, although $H$ and $G$ vanish near $\partial X$, they are not required to vanish on any fixed open set.  
   
   \begin{proof} [Proof of Theorem \ref{Main Theorem non-compact}]
   The proof of this theorem is very similar to that of Theorem \ref{Main Theorem} and thus we will not provide it in detail.  Here, we will only explain why it is not necessary to require that $H$ and $G$  vanish on any fixed open subset of $X$:  to construct spectral invariants, the authors of \cite{FS07, La11} extend $(X, \omega)$ to an open symplectic manifold $(\hat{X}, \hat{\omega})$, where $$\hat{X} = X \cup_{\partial X \times \{0\}} \partial X \times [0,\infty) \text{ and } \hat{\omega} = \left\{ \begin{array} {ll} 
                                         \omega  & \mbox{on $X$};\\
                                          d(e^r \alpha)& \mbox{on $[0, \infty) \times \partial X$}.\end{array} \right. $$
   
   \noindent Here, $r$ denotes the coordinate on $[0, \infty)$ and $\alpha= \iota_V \omega.$  Hamiltonian Floer theory is then carried out for so-called \emph{admissible Hamiltonians}; we emphasize that \emph{admissible Hamiltonians} are Hamiltonians on $\hat{X}$ and can be non-zero on  $\partial X \times [0,\infty)$.  Spectral invariants are then constructed for \emph{admissible Hamiltonians} as in Section \ref{Spectral Invariants}.   Elements of $C^{\infty}_c([0,1] \times X)$, viewed as Hamiltonians vanishing on $\partial X \times [0,\infty)$, are admissible and hence one can associate spectral invariants to them.
   
   Now, let $Y = X \cup_{\partial X \times \{0\}} \partial X \times [0,1]$.  We view $Y$ as a symplectic manifold with boundary; the symplectic form is taken to be $\hat{\omega}|_Y$. As explained in the previous paragraph, spectral invariants can be constructed for elements of $C^{\infty}_c([0,1] \times Y)$.  Suppose that $H \in C^{\infty}_c([0,1] \times X) \subset C^{\infty}_c([0,1] \times Y)$.  It can be checked that spectral invariants of $H$ are independent of whether $H$ is viewed as a Hamiltonian on $X$ or $Y$. Hence, we will view elements of $C^{\infty}_c([0,1] \times X)$ as functions on $Y$ that vanish on $\partial X \times (0,1)$ and the role of the open set $U$ from Theorem \ref{Main Theorem} will be played by $\partial X \times (0,1)$.  The set $X$ can be $\epsilon$-shifted inside $Y$ in the same way that $M\setminus U$ was $\epsilon$-shifted inside $M$ in the proof of Theorem \ref{Main Theorem}: using a $C^{\infty}$-small time-independent Hamiltonian whose critical points are all contained in $\partial X \times (0,1).$  The rest of the proof parallels the proof of Theorem \ref{Main Theorem}.  
   \end{proof}
  
  \subsection{Infiniteness of Hofer's diameter and Le Roux's question on compact manifolds with boundary}   
   Theorem \ref{Main Theorem non-compact} implies the following:
   Suppose that there exists $H \in C^{\infty}_c([0,1]\times X)$ with arbitrarily large spectral invariants.  Then,
   
    \begin{enumerate}
   \item  $Ham_c(X)$ has infinite Hofer diameter. 
   \item If the Calabi homomorphism descends to $Ham_c(X)$, the kernel of $Cal$ has infinite Hofer diameter. 
   \item Le Roux's question is answered affirmatively, i.e., $\mathcal{E}_A(X)$ has non-empty $C^0$ interior for any value of $A>0$.
   \end{enumerate}
   
   Proofs of the above facts are omitted because of their similarity to those of Theorems \ref{Th: infiniteness of diameter} and \ref{Th: Le Roux's question}.  
   
   Spectral invariants are unbounded on many manifolds; for instance this is true, as explained in Example 5.8 of \cite{Sc00}, if $X$ contains a Lagrangian $L$ such that $\pi_1(L)$ embeds into $\pi_1(X)$ and $L$ admits a Riemannian metric with no non-constant contractible geodesics.  
  
\printbibliography

\end{document}